\documentclass[11pt,reqno,tbtags,a4paper]{amsart}
\usepackage{amssymb}
\usepackage{xpunctuate}
\usepackage{url}
\usepackage[square,numbers]{natbib}
\bibpunct[, ]{[}{]}{;}{n}{,}{,}

\title%[]
{Rate of convergence for traditional P\'olya urns }

\date{21 November, 2019}
\author{Svante Janson}
\address{Department of Mathematics, Uppsala University, PO Box 480,
SE-751~06 Uppsala, Sweden}
\email{svante.janson@math.uu.se}
\urladdr{http://www.math.uu.se/svante-janson}

\subjclass[2010]{} 
\numberwithin{equation}{section}

\renewcommand\le{\leqslant}
\renewcommand\ge{\geqslant}

\allowdisplaybreaks

\theoremstyle{plain}% default
\newtheorem{theorem}{Theorem}[section]
\newtheorem{lemma}[theorem]{Lemma}

\newtheorem{corollary}[theorem]{Corollary}

\theoremstyle{definition}

\newtheorem{remark}[theorem]{Remark}

\theoremstyle{remark}

\newenvironment{romenumerate}[1][-10pt]{% optional argument changes indentation
\addtolength{\leftmargini}{#1}\begin{enumerate}% gives (i), (ii) etc.
 \renewcommand{\labelenumi}{\textup{(\roman{enumi})}}%
 }{\end{enumerate}}

\newcounter{oldenumi}
{\setcounter{oldenumi}{\value{enumi}}
\begin{romenumerate} \setcounter{enumi}{\value{oldenumi}}}
{\end{romenumerate}}

\newcounter{thmenumerate}
\newenvironment{thmenumerate}
{\setcounter{thmenumerate}{0}%
 \def\item{\par% \ifnum\thethmenumerate=0\else\par\fi %\noindent\fi
 \refstepcounter{thmenumerate}\textup{(\roman{thmenumerate})\enspace}}
}
{}

\newcounter{xenumerate}   %no left indentation; thus wider lines

\newcommand{\refT}[1]{Theorem~\ref{#1}}
\newcommand{\refTs}[1]{Theorems~\ref{#1}}
\newcommand{\refC}[1]{Corollary~\ref{#1}}
\newcommand{\refL}[1]{Lemma~\ref{#1}}
\newcommand{\refR}[1]{Remark~\ref{#1}}
\newcommand{\refS}[1]{Section~\ref{#1}}
\newcommand{\refSs}[1]{Sections~\ref{#1}}
\newcommand\marginal[1]{\marginpar[\raggedleft\tiny #1]{\raggedright\tiny#1}}
\newcommand\SJ{\marginal{SJ} }

\newcommand\REM[1]{{\raggedright\texttt{[#1]}\par\marginal{XXX}}}
\newcommand\XREM[1]{\relax}

\begingroup
  \divide\count255 by 60
  \multiply\count255 by -60
  \advance\count255 by \time
  \ifnum \count255 < 10 \xdef\klockan{\the\count1.0\the\count255}
  \else\xdef\klockan{\the\count1.\the\count255}\fi
\endgroup

\newcommand{\sumiok}{\sum_{i=0}^k}
\newcommand{\sumiq}{\sum_{i=1}^q}

\newcommand\set[1]{\ensuremath{\{#1\}}}
\newcommand\bigset[1]{\ensuremath{\bigl\{#1\bigr\}}}
\newcommand\Bigset[1]{\ensuremath{\Bigl\{#1\Bigr\}}}

\newcommand\xpar[1]{(#1)}
\newcommand\bigpar[1]{\bigl(#1\bigr)}
\newcommand\Bigpar[1]{\Bigl(#1\Bigr)}
\newcommand\biggpar[1]{\biggl(#1\biggr)}

\newcommand\abs[1]{|#1|}
\newcommand\bigabs[1]{\bigl|#1\bigr|}

\def\rompar(#1){\textup(#1\textup)}    % usage: \rompar(...)
\newcommand\xfrac[2]{#1/#2}
\newcommand\xpfrac[2]{(#1)/#2}

\newcommand\Bigparfrac[2]{\Bigpar{\frac{#1}{#2}}}

\newcommand\xparfrac[2]{\xpar{\xfrac{#1}{#2}}}

\def\xexp(#1){e^{#1}}
\newcommand\ceil[1]{\lceil#1\rceil}
\newcommand\floor[1]{\lfloor#1\rfloor}

\newcommand\frax[1]{\{#1\}}
\newcommand\ntoo{\ensuremath{{n\to\infty}}}

\newcommand\norm[1]{\|#1\|}

\newcommand\upto{\nearrow}
\newcommand\punkt{\xperiod}    % xpunctuate
\newcommand\iid{i.i.d\punkt}    
\newcommand\ie{i.e\punkt}
\newcommand\eg{e.g\punkt}

\newcommand{\as}{a.s\punkt}

\newcommand\ii{\mathrm{i}}

\newcommand{\tend}{\longrightarrow}
\newcommand\dto{\overset{\mathrm{d}}{\tend}}

\newcommand\eqd{\overset{\mathrm{d}}{=}}

\newcommand\bbR{\mathbb R}

\newcommand\bbT{\mathbb T}

\newcommand\bbZ{\mathbb Z}

\newcounter{CC}
\newcounter{cc}
\newcommand\E{\operatorname{\mathbb E{}}}
\renewcommand\P{\operatorname{\mathbb P{}}}

\newcommand\Bi{\operatorname{Bi}}

\newcommand\ga{\alpha}
\newcommand\gb{\beta}
\newcommand\gd{\delta}
\newcommand\gD{\Delta}
\newcommand\gf{\varphi}

\newcommand\gG{\Gamma}

\newcommand\gO{\Omega}

\newcommand\eps{\varepsilon}

\renewcommand\phi{\xxx}  %% WARNING

\newcommand\cL{{\mathcal L}}

\newcommand\cS{{\mathcal S}}

\newcommand\qw{^{-1}}
\newcommand\qww{^{-2}}

\newcommand\qqw{^{-1/2}}

\newcommand\dd{\,\mathrm{d}}

\newcommand\vektor{\mathbf}
\newcommand\vX{\vektor{X}}
\newcommand\vx{\vektor{x}}
\newcommand\vY{\vektor{Y}}

\newcommand\vV{\vektor{V}}
\newcommand\vW{\vektor{W}}
\newcommand\vWx{\vektor{W}^*}
\newcommand\ve{\vektor e}
\newcommand\bbRp{\bbR_{>0}}
\newcommand\citetyear[1]{\citet{#1} (\citeyear{#1})}
\newcommand\Dir{\mathrm{Dir}}
\newcommand\Beta{\mathrm{Beta}}
\newcommand\nk{_{n,k}}
\newcommand\yn{Y_{n,1}}
\newcommand\gagb{_{\ga,\gb}}
\newcommand\nn{\floor{n/2}}
\newcommand\ynx{\yn^*}
\newcommand\Zx{Z^*}
\newcommand\Ooo{O_{L_\infty}}
\newcommand\normoo[1]{\norm{#1}_{L_\infty}}
\newcommand\CK{K}
\newcommand\CKK{K_1}
\newcommand\ellx[1]{\ell_{#1}}
\newcommand\ellp{\ellx{p}}
\newcommand\elloo{\ellx{\infty}}
\newcommand\dw{d_{\rm W}}
\newcommand\dks{d_{\rm{KS}}}
\newcommand\dl{d_{\rm{L}}}
\newcommand\KS{Kolmogorov--Smirnov}

\newcommand{\Polya}{P\'olya}

\newcommand{\Levy}{L\'evy}

\begin{document}

\begin{abstract} 
Consider a P\'olya urn with balls of several colours, where balls are drawn
sequentially 
and each drawn ball immediately 
is replaced together with a fixed number of balls of the
same colour. It is well-known that the proportions of balls of the different
colours converge in distribution to a Dirichlet distribution. We show that
the rate of convergence is $\Theta(1/n)$ in the minimal $L_p$ metric for any
$p\in[1,\infty]$, extending a result by Goldstein and Reinert; we further
show the same rate for the \Levy{} distance, while the rate for the
Kolmogorov distance depends on the parameters, \ie, on the
initial composition of the urn.
The method  used here differs from the one used by Goldstein and Reinert,
and uses direct calculations based on the known exact distributions.
\end{abstract}

\maketitle

\XREM{
Write Abstract!
Update date!
No showkeys!
}

\section{Introduction}\label{S:intro}

A (traditional) \Polya{} urn contains balls of several colours.
At discrete time steps, a ball is drawn
at random from the urn, and it is replaced together with $a$ balls of the
same colour, where $a>0$ is some given constant.
These urn models 
were studied already in \citeyear{Markov1917} by \citet{Markov1917}, 
but they 
are named after George \Polya, who studied them in
\citetyear{EggPol1923} and \citetyear{Polya1930}.
(These early references studied the case $q=2$; the extension to general
$q$ is straightforward.
See also \eg{}
\cite[Chapter 4]{JohnsonKotz} and \cite{Mahmoud}.)

Let $q$ be the number of colours 
and let  
the vector $\vX_n=(X_{n,1},\dots,X_{n,q})$ ($n\ge0$)
describe the numbers of balls
of the different colours at time $n$. 
(We  assume that the colours are the integers $1,\dots,q$.)
The description above thus means that the random vectors $\vX_n$
evolve as a Markov chain, with some given (deterministic) initial vector
$\vX_0=\vx_0=(x_1,\dots,x_q)$ and 
\begin{equation}\label{polya}
  \P\bigpar{\vX_{n+1}=\vX_n+a\ve_i\mid\vX_n}
=\frac{X_{n,i}}{|\vX_n|},
\qquad i=1,\dots,q,
\end{equation}
where
$\ve_i=(\gd_{ij})_{j=1}^q$ are the unit vectors and
 $|\vX_n|:=\sumiq X_{n,i}$.
Note that the total number of balls $|\vX_n|=an+|\vx_0|$ is deterministic.

Although the formulation in the first paragraph above talks about 'the
number of balls', thus implying that the numbers are integers, we see that
no such assumption is needed in the more formal definition using \eqref{polya}.
Hence, from now on, $a$ and $x_1,\dots,x_q$
may be arbitrary positive real numbers, and thus the random vectors
$\vX_n\in\bbRp^q$. 
We assume that $a$ and $x_1,\dots,x_q$ are strictly positive to avoid
trivial complications; this implies that $X_{n,i}\ge x_i>0$ for all $n\ge0$
and $i\in[q]:=\set{1,\dots,q}$. In particular, $|\vX_n|>0$, and
the transition probabilities are well defined by 
\eqref{polya}.

One advantage of the extension to real values is that the model is
homogeneous in the sense that if we multiply all $X_n$ (including the
initial values) and $a$ by the same positive real number, then \eqref{polya}
still holds, so we have another \Polya{} urn.
In particular, by replacing $\vX_n$ by $\vX_n/a$, we may without loss of
generality assume that $a=1$.

We define also the random vector $\vY_n=(Y_{n,1}.\dots,Y_{n,q})$, where
$Y_{n,i}$ is the number of the first $n$ drawn balls that have colour $i$.
Obviously,
\begin{equation}\label{xy}
  \vX_n=\vx_0+a\vY_n,
\end{equation}
so it is equivalent to study $\vX_n$ and $\vY_n$.

It is easy to find the exact distribution of $\vY_n$ and thus of $\vX_n$,
see \eg{} \cite{Markov1917,EggPol1923,Polya1930,JohnsonKotz,Mahmoud}
and \eqref{pni} below, and from this it is easy to see that 
as \ntoo,
the  fraction of balls of a given colour converges in distribution 
to  a Beta distribution \cite{Polya1930}; more precisely
\begin{equation}\label{limbeta}
\frac{X_{n,i}}{|\vX_n|}\dto \Beta\Bigpar{x_i/a,\sum_{j\neq i}x_j/a}.
\end{equation}
 More generally, the vector $\vX_n/|\vX_n|$ of proportions converges to a
Dirichlet distribution
(see \refS{Snot} for the definition):
\begin{equation}\label{limdir}
\frac{\vX_{n}}{|\vX_n|}\dto \Dir\xpar{\vx_0/a}.
\end{equation}
It follows by \eqref{xy} that the same holds for the
proportions $\vY_n/|\vY_n|$.

\begin{remark}
  In fact, it is easy to see that $\vX_n/|\vX_n|$ is a martingale, and thus 
$\vX_n/|\vX_n|$ converges \as{} to a limit, which thus has the Dirichlet
distribution $\Dir\xpar{\vx_0/a}$.
It follows by \eqref{xy} that the same holds for  $\vY_n/|\vY_n|$.
(The \as{} convergence can also be seen in other ways, for example by de
Finetti's theorem 
and the fact that the sequence of colours of the drawn balls is exchangeable,
see \refR{Ras},
or by Martin boundary theory \cite{BlackwellKendall1964,Sawyer}.)
\end{remark}

The purpose of the present paper is to study the rate of convergence 
of the distributions in 
\eqref{limbeta} and \eqref{limdir}.
For the Wasserstein metric $\dw$, this was done by \citet{GoldsteinReinert},
who proved, in the case $q=2$, a bound of the order $O(1/n)$, with an explicit
constant depending on $a$ and $\vx_0$.

The Wasserstein metric equals the minimal $L_1$ metric $\ellx1$, 
and our main result is the following, which
extends the result by \citet{GoldsteinReinert} to $\ellp$ for all
$p\le\infty$, and to all numbers of colours $q\ge2$.
See \refS{Snot} for definitions and \refSs{Spf2}--\ref{Spfq} for proofs.

\begin{theorem}
  \label{T1}
For any $q\ge2$, any $a>0$ and any initial values
$\vX_0=\vx_0%=(x_1,\dots,x_q)
\in\bbRp^q$, 
%there exists constants $c=c(q,a,\vx_0)$ and $C=C(q,a,\vx_0)$
%such that 
if\/ $\vW\sim\Dir\bigpar{\vx_0/a}$, then
for 
every $p\in[1,\infty]$ and
all $n\ge1$,
\begin{align}
\ellp\biggpar{\frac{\vX_n}{|\vX_n|},\vW}
&=\Theta\Bigparfrac1{n},\label{ax}
%\frac{c}n &\le  \ellp\biggpar{\frac{\vX_n}{|\vX_n|},\vW}
%\le \frac{C}{n},\label{ax}
\\
\ellp\biggpar{\frac{\vY_n}{n},\vW}
&=\Theta\Bigparfrac1{n}.\label{ay}
%\frac{c}{n}&\le  \ellp\biggpar{\frac{\vY_n}{n},\vW} \le \frac{C}{n}.\label{ay}
%  \elloo\bigpar{\vY_n/n,\vW} &\le \frac{C}{n}.
\end{align}
\end{theorem}

\begin{corollary}
  \label{C1}
For any $q\ge2$, any $a>0$ and any initial values
$\vX_0=\vx_0%=(x_1,\dots,x_q)
\in\bbRp^q$, 
%there exists constants $c$ %=c(q,a,\vx_0)$ and $C$ %=C(q,a,\vx_0)$ such that 
if\/ $W_i\sim\Beta\bigpar{x_i/a,(|\vx_0|-x_i)/a}$, then
for 
every $p\in[1,\infty]$, every $i\in[q]$ and
all $n\ge1$,
\begin{align}
\ellp\biggpar{\frac{X_{n,i}}{|\vX_n|},W_i}
&=\Theta\Bigparfrac1{n},
\label{axb}
\\
\ellp\biggpar{\frac{Y_{n,i}}{n},W_i}
&=\Theta\Bigparfrac1{n},
\label{ayb}
%  \elloo\bigpar{\vY_n/n,\vW} &\le \frac{C}{n}.
\end{align}
\end{corollary}

We prove \refT{T1} by induction on $q$. The main part of the proof is the
base case $q=2$, which is proved in \refS{Spf2}.
The easy induction for general $q$ is done in \refS{Spfq}.

The proof by \citet{GoldsteinReinert} is based on Stein's method, where they
develop a version for the Beta distribution. In contrast, the present paper
uses only direct calculations, based on the known exact distributions.

Since the $\ellp$ metrics are monotone in $p$, it suffices to prove the
lower bounds for $p=1$ and the upper bounds for $p=\infty$.
Moreover,
%\begin{remark}
  the lower bounds in \eqref{ax}--\eqref{ayb} are trivial, since
$W_i$ has a  continuous distribution with continuous density function,
while $X_{n,i}/|\vX_n|$ and $Y_{n,i}/n$ are discrete with values spaced by 
$a/(an+|\vx_0|)$ and $1/n$, respectively, see \refS{Spf2} for details.
%\end{remark}
Since $|\vX_n|=an+|\vx_0|$, 
the upper bounds in \eqref{ax} and \eqref{ay} are for $p=\infty$
equivalent to,
respectively, 
\begin{align}
  \elloo\bigpar{{\vX_n},|\vX_n|\vW} &=O(1),\label{a2x}
\\
  \elloo\bigpar{{\vY_n},n\vW} &=O(1).\label{a2y}
\end{align}
Moreover, by \eqref{xy}, \eqref{a2x} is equivalent to  
$  \elloo\bigpar{a\vY_n,an\vW} =O(1)$ and thus to \eqref{a2y}.
Hence the upper bounds in the
two assertions \eqref{ax} and \eqref{ay} in the theorem are
equivalent. 

For the one-dimensional variables $X_{n,i}$ and $Y_{n,i}$,
we give also the corresponding results for the Kolmogorov--Smirnov
distance $\dks$ and the \Levy{} distance $\dl$. 
The proofs are given in \refS{SKS}.

\begin{theorem}  \label{TKS}
  With assumptions and notations as in \refC{C1},
and 
\begin{align}\label{rho}
\rho:=\min\Bigset{\frac{x_i}a,\frac{|\vx_0|-x_i}a,1},   
\end{align}
we have
\begin{align}%\label{tks}
\dks\biggpar{\frac{X_{n,i}}{|\vX_n|},W_i}
&=\Theta\bigpar{ n^{- \rho}},
\label{tksx}
\\
\dks\biggpar{\frac{Y_{n,i}}{n},W_i}
&=\Theta\bigpar{ n^{- \rho}}.
\label{tksy}
\end{align}
\end{theorem}

\begin{theorem}  \label{TL}
  With assumptions and notations as in \refC{C1},
\begin{align}%\label{tl}
\dl\biggpar{\frac{X_{n,i}}{|\vX_n|},W_i}
&=\Theta\bigpar{ n^{- 1}},
\label{tlx}
\\
\dl\biggpar{\frac{Y_{n,i}}{n},W_i}
&=\Theta\bigpar{ n^{- 1}}.
\label{tly}
\end{align}
\end{theorem}

%\marginal{Prohorov?}

\begin{remark}\label{RKS}
  Note that the rate for the \KS\ distance differs from the other distances
  considered here, but only in the case when one of the parameters 
$\ga:=\xfrac{x_i}a$ and $\gb:=\xpfrac{|\vx_0|-x_i}a$ of
$W_i\sim\Beta(\ga,\gb)$ is less that 1.
It follows from the proofs below that this difference can be attributed to
the fact that the density \eqref{fab} of $W_i$ is unbounded in (precisely)
this case.
\end{remark}

\begin{remark}
  We do not attempt to find explicit values for the constant  $C$ in the upper
bounds above; 
although in principle it should be possible by careful calculations
in the proof.
Note that the proof below treats the cases $x_i<a$, $x_i=a$ and $x_i>a$
separately, and the constants implicit in the calculations might blow up
as $x_i\to a$ from above or below. 
We conjecture that
the constant $C$ can be chosen uniformly for, say, all  $\vx_0$ with
$a/2\le x_i\le 2a$, 
but it is not obvious whether this follows from (some version of)
our method of proof or not,
and we leave that as an open problem.
\end{remark}

\begin{remark}\label{Ras}
The present paper studies the rate of convergence of the \emph{distribution} of
  $\vX_n/|\vX_n|$. As said above, this sequence of random variables
  converges \as{} to a limit $\vW$, and one might also ask about the \as{} 
rate of   convergence of the variables.
In fact, 
the sequence of  colors of the drawn balls is exchangeable,
which (as noted by \cite{Polya1930}) follows from the simple formula for the
probability distribution of the sequence (see \eqref{pni} and the comments
after it).
Hence, 
de Finetti's theorem shows that,
conditioned on the \as{}
limit $\vW$, the sequence $\vY_n$ is the sequence of partial sums of an \iid{}
sequence of random unit vectors $Z_n$, with the distribution
$\P(Z_n=\ve_i)=W_i$; in particular, each coordinate is the sequence of sums
of an \iid{} Bernoulli sequence. (See \eg \cite[Theorem 11.10]{Kallenberg}.)
Consequently, we have the same rate of
convergence for the trajectories
$\vY_n/n$ and $\vX_n/|\vX_n|$ as for the strong law of large
number for \iid{} Bernoulli sequences. Thus, by the central limit theorem,
the distance $|\vX_n/|\vX_n|-\vW|$ is typically of order $n\qqw$, and more
precisely there is a law of iterated logarithm.

Note that the distributions thus
converge faster that the trajectories. 
This method with conditioning on the limit 
can be used
to obtain an upper bound for the distance in
\eqref{ay} (for $p<\infty$, at least), 
see \cite[Lemma 2.1]{KNeininger},
but it only yields a bound
$O(n\qqw)$, the $\ellp$ rate of convergence of 
$n\qw\Bi(n,q)$ to the constant $q$,
for $q\in(0,1)$.
\end{remark}

\begin{remark}
  Generalized versions of the urn model above, where we may add to the urn
  also balls of colours
  other than the colour of the drawn one, have a rich theory with,
  typically, quite different behaviour. See \eg{} \cite{SJ154}.
Such urns will not be studied in the present paper.
\end{remark}

\begin{remark}
  The initial motivation for this study was a paper by \citet{KNeininger},
  where generalized \Polya{} urns (with other replacement schemes) are
  considered; one step in the proof in \cite{KNeininger} was a weaker
  version of \eqref{ayb} with the bound $O(n\qqw)$. I initially hoped that
  the sharper result in \refC{C1} would lead to an improvement of the
  result in \cite{KNeininger} for generalized \Polya{} urns, but it turns
  out that this is not the case; 
the estimate found in \cite{KNeininger} 
of the convergence rate is dominated by other terms in their proof. 
\end{remark}

\section{Notation and other preliminaries}\label{Snot}

\subsection{General}
For a vector $\vx=(x_1,\dots,x_q)\in\bbR^q$, we define
$|\vx|:=\sum_i|x_i|$. (This is mainly for convenience, in for exampe
\eqref{polya}. When defining the metrics below, the Euclidean distance would
be just as fine, within some constant factors in the result.)

$\floor{x}$ denotes the integer part of a real number $x$, and
 $\frax{x}:=x-\floor x$ the fractional part.

$F_X(x):=\P(X\le x)$ denotes the distribution function of a real-valued
random variable $X$.

$C$ and $c$ denote various positive 
constants that may depend on the parameters 
$q$, $a$ and $\vx_0$, 
but not on $n$; they may change from one
occurrence to another.

For sequences $a_n$ and $b_n$ of real numbers, with $b_n>0$,
$a_n=O(b_n)$ means $|a_n|\le C b_n$, for some $C$.
Similarly, 
$a_n=\Theta(b_n)$ means $cb_n\le |a_n|\le C b_n$.
%As just said, the implicit constants may

\subsection{Beta and Dirichlet distributions}
Recall that the \emph{Beta distribution} 
$\Beta(\ga,\gb)$, where $\ga,\gb>0$, has
the density function 
\begin{equation}\label{fab}
  f_{\ga,\gb}(x)=B(\ga,\gb)\qw x^{\ga-1}(1-x)^{\gb-1},
\qquad (0<x<1)
\end{equation}
where the normalizing constant is given by the \emph{Beta function}
\begin{equation}
  B(\ga,\gb):=\frac{\gG(\ga)\gG(\gb)}{\gG(\ga+\gb)}.
\end{equation}

Recall also that if $q\ge2$ and $\ga_1,\dots,\ga_q>0$, then the
\emph{Dirichlet distribution} $\Dir(\ga_1,\dots,\ga_q)$ is the distribution
on the simplex $\bigset{(x_1,\dots,x_q)\in\bbRp^q:\sum_i x_i=1}$ with density
\begin{equation}
\frac{\gG(\ga_1+\dots+\ga_q)}{\gG(\ga_1)\dotsm\gG(\ga_q)}\,
  x_1^{\ga_1}\dotsm x_q^{\ga_q} \dd x_1\dotsm \dd x_{q-1}.
\end{equation}
In the case $q=2$, this is essentially the same as a Beta distribution; more
precisely $\vW\in\Dir(\ga,\gb)$ if and only if $\vW=(V,1-V)$, with
$V\in\Beta(\ga,\gb)$. More generally, if
$(W_1,\dots,W_q)\in\Dir(\ga_1,\dots,\ga_q)$, 
then $W_i\sim\Beta(\ga_i,\ga'_i)$, with $\ga'_i:=\sum_{j\neq i}\ga_j$.

One well-known construction of Dirichlet distributed random variables 
is that if $V_1,\dots,V_q$ are independent random variables with
$V_i\sim\Gamma(\ga_i)$, then the vector of proportions
\begin{equation}\label{dirq}
  \frac{(V_1,\dots,V_q)}{\sum_i V_i}\sim\Dir(\ga_1,\dots,\ga_q),
\end{equation}
and this vector is independent of $\sum_{i=1}^q V_i$.
If $q>2$, then, 
letting $V':=\sum_{i=2}^q V_i$,
\begin{equation}
  \frac{(V_1,\dots,V_q)}{\sum_i V_i}
=\Bigpar{\frac{V_1}{V_1+V'},\frac{V'}{V_1+V'}\frac{(V_2,\dots,V_q)}{V'}},
\end{equation}
and it follows from \eqref{dirq} that if $Z\sim\Beta(\ga_1,\ga')$ and
$\vV\sim\Dir(\ga_2,\dots,\ga_q)$ are independent, 
with $\ga'=\sum_{i=2}^q \ga_i$, then 
\begin{equation}\label{dirr}
  \bigpar{Z,(1-Z)\vV}\sim\Dir(\ga_1,\dots,\ga_q).
\end{equation}

\subsection{Some probability metrics}

A probability metric is a metric on a suitable set of probability
distributions in some space $\cS$; in this paper we only consider $\cS=\bbR$
or $\bbR^q$ for some fixed $q$. 
(See \eg{} \cite{Rachev} for a general theory and many examples, including
the ones below.)
Although a probability metric $d(\mu,\nu)$ is
formally defined for distributions $\mu$ and $\nu$, we follow common
practice and write  $d(X,Y):=d(\cL(X),\cL(Y))$ when $X$ and $Y$ are random
variables with distributions $\cL(X)$ and $\cL(Y)$, and  we 
state  the definitions below in this form.

In the present paper, we mainly study the \emph{minimal $L_p$ distance} $\ellp$
for $1\le p\le\infty$,
defined by
\begin{align}
\label{ellp}
\ellp(X,Y):=\inf\bigset{\norm{X'-Y'}_p:X'\eqd X,\,Y'\eqd Y},
\end{align}
where $\norm{X'-Y'}_p$ is the usual $L_p$ distance, \ie,
$(\E|X'-Y'|^p)^{1/p}$ for $p<\infty$, and the essential supremum of
$|X'-Y'|$ if $p=\infty$,
and the infimum is taken ovar all \emph{couplings} of $X$ and $Y$,
\ie, all pairs $(X',Y')$ of random variables on a common probability space
with the same (marginal) distributions as $X$ and $Y$.
(The infimum is actually attained, see  \cite[Corollary 5.3.2]{Rachev}.)  
%\cite{Cambanis} 
This is a metric on the
set of all probability distributions in $\bbR^q$
with finite $p$th absolute moment.

The metric $\ellx1$ is also known as the
\emph{Wasserstein distance} %(after L.N. Vasershtein)
and the \emph{Kantorovich distance}; 
see \cite{Wasserstein-metric} for a history.

For random variables $X$ and $Y$ in $\bbR$, we consider also the
\emph{\KS\ distance} 
\begin{align}\label{dks}
  \dks(X,Y):=
\sup_{x \in \bbR} |F_X(x) - F_Y(x)| 
=
\sup_{x \in \bbR} |P(X \leq x) - P(Y \leq x)| 
\end{align}
and the \emph{\Levy\ distance}
\begin{align}\label{dl}
  \dl(X,Y):=
\inf\bigset{\eps>0:
F_X(x-\eps)-\eps \le  F_Y(x) 
\le F_X(x+\eps)+\eps
\text{ for all }x}.
\end{align}

%$\dw=\ellx1$

\section{Proof of \refT{T1}: $q=2$}\label{Spf2}\SJ

As explained in the introduction, we may without loss of generality assume
$a=1$, and we shall do so throughout the proof. 
For convenience, 
we call the two colours \emph{white} and \emph{black}, and
we  write $x_1=\ga$ and $x_2=\gb$.
We thus assume $a=1$ and $\vx_0=(\ga,\gb)$, and note that then
$|\vX_n|=n+\ga+\gb$. 
Furthermore,  $\vW=(W_1,W_2)=(W_1,1-W_1)$, where $W_1$
is a random variable with the distribution
$\Beta(\ga,\gb)$, 

The parameters $\ga$ and $\gb$ are fixed throughout the section.
Recall that $C$ and $c$
denote various constants that may depend on $\ga$ and $\gb $, 
but not on $n$.
%and may change from one occurence to another. %CCC

Define, recalling that $Y_{n,1}$ is the number of white balls drawn in the
first $n$ draws, 
\begin{align}
  p_{n,i}&:=\P(Y_{n,1}=i),\label{pni}
\\
P_{n,k}&:=\P(Y_{n,1}\le k)=\sumiok p_{n,i},\label{Pnk}
\end{align}

The basis of the proof is the well-known exact formula
\cite{Markov1917,EggPol1923,Polya1930,JohnsonKotz,Mahmoud}
\begin{equation}\label{pni}
  \begin{split}
	  p_{n,i}&=\binom{n}{i} 
\frac{\ga(\ga+1)\dotsm(\ga+i-1)\gb(\gb+1)\dotsm(\gb+n-i-1)}
{(\ga+\gb)(\ga+\gb+1)\dotsm(\ga+\gb+n-1)}
\\&
=\frac{n!}{i!\,(n-i)!}\frac{\gG(\ga+i)\gG(\gb+n-i)\gG(\ga+\gb)}
{\gG(\ga)\gG(\gb)\gG(\ga+\gb+n)}
\\&
=B(\ga,\gb)\qw
\frac{\gG(i+\ga)}{\gG(i+1)}
\frac{\gG(n-i+\gb)}{\gG(n-i+1)}
\frac{\gG(n+1)}{\gG(n+\ga+\gb)}
%\frac{\gG(i+\ga)\gG(n-i+\gb)\gG(n+1)}{\gG(i+1)\gG(n-i+1)\gG(n+\ga+\gb)}
  \end{split}
\end{equation}
for $i=0,\dots,n$.
This formula is easily verified:
if we draw $i$ white balls in the first $n$ draws, then this can be done in
$\binom ni$ ways, and it follows from the definition \eqref{polya} of the
urn process that each of them has the same probability given by the fraction
on the first line of \eqref{pni}. The rest is simple calculations with Gamma
functions. 
For future use we note the special cases 
\begin{align}
  p_{n,0}&=\frac{\gG(\ga+\gb)}{\gG(\gb)}\frac{\gG(n+\gb)}{\gG(n+\ga+\gb)},
\label{pn0}
\\
  p_{n,n}&=\frac{\gG(\ga+\gb)}{\gG(\ga)}\frac{\gG(n+\ga)}{\gG(n+\ga+\gb)}.
\label{pnn}
\end{align}

%Furthermore, 
%let $W_1$ be a random variable with the distribution
Recall that $W_1\sim\Beta(\ga,\gb)$ has the density function  $f\gagb$  
given in \eqref{fab} 
and define
\begin{align}
  Q_{n,k}&:=\P(W_1\le k/n) = \int_{0}^{k/n}f_{\ga,\gb}(x)\dd x  \label{Qnk}
\intertext{and}
R_{n,k}&:=P\nk-Q\nk=\P(\yn\le k)-\P(W_1\le k/n).\label{Rnk}
\end{align}

We also define $\gD P\nk:=P\nk-P_{n,k-1}=p\nk$, $\gD Q\nk:=Q\nk-Q_{n,k-1}$
and $\gD R\nk:=R\nk-R_{n,k-1}$.
\begin{lemma}
  \label{L1}
  \begin{thmenumerate}
  \item \label{L1a}
If\/ $1\le k\le 3n/4$, then
\begin{equation}
  \label{l1a}
|\gD R\nk|\le C \xfrac{k^{\ga-2}}{n^\ga}.
\end{equation}
If\/ further $\ga=1$, this can be improved to
\begin{equation}
  \label{l1a1}
|\gD R\nk|\le C n\qww.
\end{equation}
  \item \label{L1b}
If\/ $n/4\le k\le n$, then
\begin{equation}
  \label{l1b}
|\gD R\nk|\le C \xfrac{(n-k+1)^{\gb-2}}{n^\gb}.
\end{equation}
If\/ further $\gb=1$, this can be improved to
\begin{equation}
  \label{l1b1}
|\gD R\nk|\le C n\qww.
\end{equation}
  \end{thmenumerate}
\end{lemma}

\begin{proof}
  We recall the standard formula (an easy consequence of Stirling's formula,
  see also \cite[5.11.13]{NIST}),
  \begin{equation}\label{GG}
	\frac{\gG(x+a)}{\gG(x+b)} = x^{a-b}\Bigpar{1+O\Bigparfrac{1}{x}}
  \end{equation}
for, say, fixed $a,b\ge0$ and $x\ge1$.

Using \eqref{GG} in \eqref{pni}, we find, for $1\le k<n$,
\begin{equation}\label{gDP}
  \begin{split}
	  p\nk&= B(\ga,\gb)\qw k^{\ga-1}(n-k)^{\gb-1}n^{1-\ga-\gb}
\Bigpar{1+O\Bigparfrac{1}{k}+O\Bigparfrac{1}{n-k}+O\Bigparfrac{1}{n}
}
\\
&=
n\qw f_{\ga,\gb}\Bigparfrac kn 
\Bigpar{1+O\Bigparfrac{1}{k}+O\Bigparfrac{1}{n-k} }.  
\end{split}
\raisetag{\baselineskip}
\end{equation}

Furthermore, if $2\le k<n$, then, for $x\in((k-1)/n,k/n]$,
\begin{equation}
  \frac{f\gagb(x)}{f\gagb(k/n)}
=\Bigparfrac{nx}{k}^{\ga-1}\Bigparfrac{n-nx}{n-k}^{\gb-1}
={1+O\Bigparfrac{1}{k}+O\Bigparfrac{1}{n-k} }.  
\end{equation}
and hence
\begin{equation}\label{gDQ}
  \gD Q_{n,k}=\int_{(k-1)/n}^{k/n}f\gagb(x)\dd x
=n\qw f\gagb\Bigparfrac{k}n
\Bigpar{1+O\Bigparfrac{1}{k}+O\Bigparfrac{1}{n-k} }.  
\end{equation}
Moreover, for $k=1$, we have, crudely,
\begin{equation}\label{gDQ1}
  \gD Q_{n,1}=\int_{0}^{1/n}f\gagb(x)\dd x
\le C \int_{0}^{1/n} x^{\ga-1}\dd x
=C n^{-\ga}
=n\qw f\gagb\Bigparfrac1nO(1),
\end{equation}
and thus \eqref{gDQ} holds also for $k=1$.

Combining \eqref{gDP} and \eqref{gDQ} we thus obtain, for $1\le k< n$,
\begin{equation}\label{gDR}
  \begin{split}
\gD R\nk=\gD P\nk-\gD Q\nk
=
n\qw f_{\ga,\gb}\Bigparfrac kn 
\Bigpar{O\Bigparfrac{1}{k}+O\Bigparfrac{1}{n-k} }.  
\end{split}
\end{equation}
If we further assume $1\le k\le 3n/4$, then $f\gagb(k/n)\le C(k/n)^{\ga-1}$,
and \eqref{l1a} follows.
Similar calculations, or interchanging the colours and using
%black--white symmetry and 
\eqref{l1a},
yield \eqref{l1b} for $n/4\le k\le n$.
%
%, we use \eqref{pnn} and \eqref{GG} to find
%$p_{n,n}=O(n^{-\gb})$, and similarly it is easily seen that 
%$\gD Q_{n,n}=O(n^{-\gb})$; thus
%$\gD R_{n,n}=O(n^{-\gb})$ and \eqref{l1b} holds for $k=n$ too.

Assume now $\ga=1$. Then \eqref{pni} simplifies to, using \eqref{GG},
for $k\le 3n/4$,
\begin{equation}
  p_{n,k}=\gb \frac{\gG(n-k+\gb)}{\gG(n-k+1)}\frac{\gG(n+1)}{\gG(n+1+\gb)}
=\gb (n-k)^{\gb-1}n^{-\gb}\Bigpar{1+O\Bigparfrac{1}{n}}.
\end{equation}
Similarly, still for $k\le 3n/4$, \eqref{gDQ} simplifies to, 
\begin{equation}
  \gD Q_{n,k}=\int_{(k-1)/n}^{k/n}\gb(1-x)^{\gb-1}\dd x
=\gb (n-k)^{\gb-1}n^{-\gb}\Bigpar{1+O\Bigparfrac{1}{n}}
\end{equation}
and we obtain
\begin{equation}
  \gD R_{n,k}=p\nk-\gD Q\nk
=\gb (n-k)^{\gb-1}n^{-\gb}O\Bigparfrac{1}{n}
=O\bigpar{n\qww},
\end{equation}
showing \eqref{l1a1}. The proof of \eqref{l1b1} is similar, or by
interchanging the colours.
\end{proof}

\begin{lemma}
  \label{L2}
  \begin{thmenumerate}
  \item \label{L2a}
If\/ $0\le k\le n/2$, then
\begin{equation}
  \label{l2a}
|R\nk|\le C (k+1)^{\ga-1}n^{-\ga}.
%\xfrac{(k+1)^{\ga-1}}{n^\ga}.
\end{equation}
\item \label{L2b}
If\/ $n/2\le k\le n$, then
\begin{equation}
  \label{l2b}
|R\nk|\le C (n-k+1)^{\gb-1}n^{-\gb}. 
%\xfrac{(n-k+1)^{\gb-1}}{n^\gb}.
\end{equation}
\item \label{L2c}
If\/ $1\le k\le n-1$, then
\begin{equation}\label{l2c}
  |R\nk|\le Cn\qw f\gagb\xparfrac{k}{n}.
\end{equation}
  \end{thmenumerate}
\end{lemma}
\begin{proof}
Consider first \ref{L2a}, so $k\le n/2$.
  First note that, by \eqref{pn0} and \eqref{GG},
since $Q_{n,0}=0$,
  \begin{equation}
	R_{n,0}=P_{n,0}=p_{n,0}
%=B(\ga,\gb)\qw \gG(\ga)\frac{\gG(n+\gb)}{\gG(n+\ga+\gb)}
=O\bigpar{n^{-\ga}},
  \end{equation}
so \eqref{l2a} holds for $k=0$.

We now use \refL{L1}.
If $\ga>1$, we have by \eqref{l1a},
for $1\le k\le n/2$,
\begin{equation}
  R\nk=R_{n,0}+\sum_{i=1}^k\gD R_{n,i}
=O\Bigpar{n^{-\ga}+n^{-\ga}\sum_{i=1}^k i^{\ga-2}}
=O\bigpar{n^{-\ga}k^{\ga-1}},
\end{equation}
yielding \eqref{l2a} in this case.

If $\ga=1$, we obtain \eqref{l2a} in the same way, now using \eqref{l1a1}.

If $\ga<1$, \eqref{l1a} yields for $0\le k\le \nn$
the preliminary result
\begin{equation}\label{cama}
  \begin{split}
	R_{n,k}=R_{n,\nn}-\sum_{i=k+1}^{\nn}\gD R_{n,i}
&=R_{n,\nn}+O\Bigpar{n^{-\ga}\sum_{i=k+1}^{\nn} i^{\ga-2}}
\\&
=R_{n,\nn}+O\bigpar{n^{-\ga}(k+1)^{\ga-1}}.
  \end{split}
\end{equation}
It thus remains only to show the case $k=\floor{n/2}$; 
we postpone this.

Now turn to \ref{L2b}. 
If $\gb\ge1$,
we obtain in the same way from \eqref{l1b} and \eqref{l1b1},
noting that $R_{n,n}=1-1=0$ by the definition \eqref{Rnk}, 
for $\nn\le k\le n$,
\begin{equation}
  R\nk=-\sum_{i=k+1}^n\gD R_{n,i}
%=O\Bigpar{n^{-\ga}\sum_{i=1}^k i^{\ga-2}}
=O\bigpar{n^{-\gb}(n-k+1)^{\gb-1}}.
\end{equation}
If $\gb<1$, we obtain instead, for $\nn\le k\le n$,
\begin{equation}
  R\nk=R_{n,\nn}+\sum_{i=\nn+1}^k\gD R_{n,i}
%=O\Bigpar{n^{-\ga}\sum_{i=1}^k i^{\ga-2}}
=R_{n,\nn}+O\bigpar{n^{-\gb}(n-k+1)^{\gb-1}}.
\end{equation}

We have thus shown, for $0\le k\le\nn$,
\begin{equation}\label{ja}
  R_{n,k}=
  \begin{cases}
	O\bigpar{n^{-\ga}(k+1)^{\ga-1}},&\ga\ge1,
\\
	R_{n,\nn}+O\bigpar{n^{-\ga}(k+1)^{\ga-1}},&\ga<1,
  \end{cases}
\end{equation}
and for
$\nn\le k\le n$,
\begin{equation}\label{jb}
  R_{n,k}=
  \begin{cases}
	O\bigpar{n^{-\gb}(n-k+1)^{\gb-1}},&\gb\ge1,
\\
	R_{n,\nn}+O\bigpar{n^{-\gb}(n-k+1)^{\gb-1}},&\gb<1.
  \end{cases}
\end{equation}
This proves \eqref{l2a} and \eqref{l2b} when $\ga\ge1$ and $\gb\ge1$.

Now suppose that $\ga<1$ and $\gb\ge1$. Then 
\eqref{jb} yields $R_{n,\nn}=O\bigpar{n\qw}$, 
and \eqref{ja} shows that \eqref{l2a} holds in this case too.
%(The case $\nn<k\le3n/4$ follows directly from \eqref{jb}.)

Similarly, if $\ga\ge1$ and $\gb<1$, then \eqref{ja} yields
$R_{n,\nn}=O\bigpar{n\qw}$ and \eqref{l2b} follows from \eqref{jb}.

It remains to consider the case $\ga<1$ and $\gb<1$.
In this case we sum $R\nk$ over all $k$ and obtain by \eqref{ja} and
\eqref{jb},
recalling $R_{n,n}=0$,
\begin{equation}\label{sumR}
  \begin{split}
	  \sum_{k=0}^{n-1} R\nk
&=
nR_{n,\nn}+O\Bigpar{n^{-\ga}\sum_{k=0}^{\nn} (k+1)^{\ga-1}}
+O\Bigpar{n^{-\gb}\hskip-10pt\sum_{k=\nn+1}^{n-1} (n-k)^{\gb-1}}
\\&
=nR_{n\nn}+O(1).
  \end{split}
\raisetag{\baselineskip}
\end{equation}
On the other hand, by \eqref{Rnk},
\begin{equation}\label{sumR2}
  \begin{split}
	\sum_{k=0}^{n-1} R_{n,k}
&
=\sum_{k=0}^{n-1}\bigpar{\P(\yn\le k)-\P(nW_1\le k)}
\\&
=\sum_{k=0}^{n-1}\bigpar{\P(nW_1>k)-\P(\yn> k)}
\\&
=\sum_{k=0}^{\infty}\P(\ceil{nW_1}>k)
- \sum_{k=0}^{\infty}\P(\yn> k)
\\&
=\E \ceil{nW_1}-\E \yn
=n\E W_1+O(1)-\E \yn.
  \end{split}
\end{equation}
However, $\E W_1=\ga/(\ga+\gb)$ and by \eqref{xy}, 
since $\vX_n/|\vX_n|$ is a martingale
and $|\vX_n|=n+\ga+\gb$,
\begin{equation}
  \E\yn=\E
  X_{n,1}-\ga=(n+\ga+\gb)\frac{\ga}{\ga+\gb}-\ga=n\frac{\ga}{\ga+\gb}
=n\E W_1.
\end{equation}
Hence, \eqref{sumR}--\eqref{sumR2} yield $nR_{n,\nn}+O(1)=O(1)$, and thus
$R_{n,\nn}=O\bigpar{n\qw}$ also in this case.
The proof of \eqref{l2a} and \eqref{l2b} in this case is now completed by
\eqref{ja} and \eqref{jb} as in the cases above.

This completes the proof of \ref{L2a} and \ref{L2b}.
Finally, \ref{L2c} follows from \eqref{l2a}, \eqref{l2b} and \eqref{fab}.
\end{proof}

\begin{lemma}\label{L3}
  There exists an integer $\CK$ such that for all integers $k$ and
  $n\ge1$,
  \begin{equation}
	\label{l3}
Q_{n,k-\CK}\le P_{n,k}\le Q_{n,k+\CK}.
  \end{equation}
\end{lemma}
\begin{proof}
For any integer $L\ge1$, if $L\le k\le n-2L$, then
by \eqref{fab},
\begin{equation}\label{QL}
Q_{n,k+L}- Q_{n,k}=\int_{k/n}^{k/n+L/n}f\gagb(x)\dd x
\ge \frac{L}{n} c f\gagb\Bigparfrac{k}{n}.
\end{equation}
 It follows by \eqref{QL} and 
\refL{L2}\ref{L2c} that if $L$ is chosen large enough,
then, for all $n$ and $L\le k\le n-2L$,
\begin{equation}
  Q_{n,k+L}-Q\nk \ge R\nk = P\nk-Q\nk,
\end{equation}
and thus
\begin{equation}\label{pq}
  P\nk\le Q_{n,k+L}.
\end{equation}
Furthermore, using \eqref{pq} for $k=L$ (assuming $n\ge 3L$), we see that
if $0\le k<L$, then
\begin{equation}
  P\nk\le P_{n,L}\le Q_{n,2L}\le Q_{n,k+2L}.
\end{equation}
%This holds also, trivially, for $n<2L$.
Moreover, if $k\ge n-2L$, then trivially
\begin{equation}
  P\nk \le 1=Q_{n,n}=Q_{n,k+2L}.
\end{equation}
We conclude that if $n\ge 3L$, then for any $k\ge0$, 
\begin{equation}\label{pq+}
  P\nk\le  Q_{n,k+2L}.
\end{equation}
Moreover, this is trivially true also if $k<0$;
thus \eqref{pq+} holds for all $k\in\bbZ$ when $n\ge3L$.

In the same way we see that if $L$ was chosen large enough, also
$Q\nk-Q_{n,k-L}\ge -R\nk$ whenever $2L\le k\le n-L$, and it follows that
$P\nk\ge Q_{n,k-2L}$ for all $k\in\bbZ$ when $n\ge3L$.

This proves \eqref{l3} with $\CK=3L$ for all $n\ge1$, since the case $n<3L$ is
trivial.
\end{proof}

\begin{proof}[Proof of upper bounds in \refT{T1} for $q=2$]
We use the monotone coupling of $\yn$ and $nW_1$.
Since $nW_1$ has a continuous distribution, while 
$\yn$ is discrete,
this coupling can be constructed
as follows.
Define $t_0<\dots<t_n=n$ such that
$\P(nW_1\le t_k)=P_{n,k}=\P(\yn\le k)$, and 
define the function $g:(0,n]\to\set{0,\dots,n}$ by
$g(t)=k$ if $t\in(t_{k-1},t_k]$ (with $t_{-1}:=0$).
Then 
$\P\bigpar{g(nW_1)\le k}=\P(nW_1\le t_k)=\P(\yn\le k)$ for $k=0,\dots,n$, and
thus $\ynx:=g(nW_1)$ has the same distribution as $\yn$.

By \eqref{Qnk} and \refL{L3},
\begin{equation}
  \begin{split}
	  \P(nW_1\le k-\CK)&=Q_{n,k-\CK}\le P\nk=\P(nW_1\le t_k)
\\&
\le Q_{n,k+\CK}=\P(nW_1\le k+\CK),
  \end{split}
\end{equation}
and thus $k-\CK\le t_k\le k+\CK$ for every $k$.
Hence, if $\ynx=k$, so $nW_1\in(t_{k-1},t_k]$, then
$nW_1\le t_k\le k+\CK=\ynx+\CK$, and similarly 
$nW_1>t_{k-1}\ge k-1-\CK=\ynx-\CK-1$.
Consequently, almost surely, $|nW_1-\ynx|\le \CK+1$, which yields the desired
coupling of $nW_1$ and $\yn$.
We have thus shown
\begin{equation}\label{ellooynx}
  \elloo\bigpar{\yn,nW_1}\le\normoo{\ynx-nW_1}\le \CKK:=\CK+1.
\end{equation}
Since $\vY=(\yn,n-\yn)$ and $\vW=(W_1,1-W_1)$, this coupling also shows
$  \elloo\xpar{\vY,n\vW}\le \CKK$, \ie, \eqref{a2y}.
As said in the introduction, this is by \eqref{xy} equivalent to
\eqref{a2x}, and the upper bounds in \eqref{ax}--\eqref{ay} follow.
\end{proof}

For the proof of the lower bounds, we record a simple general fact.
\begin{lemma}
  \label{LLB}
If\/ $W$ is a continuous random variable, then
$\frax{nW}\dto U(0,1)$ as \ntoo. In particular,
\begin{align}\label{llb}
\P\bigpar{\tfrac{1}4<\frax{nW}<\tfrac{3}4}\to\tfrac12\qquad\text{as \ntoo}.  
\end{align}
\end{lemma}

\begin{proof}
As is well-known, by considering the Fourier transform on $\bbT=\bbR/\bbZ$, 
$\frax{nW}\dto U(0,1)$ %as \ntoo.
is equivalent to 
\begin{align}\label{FT}
\E e^{2\pi\ii m\frax{nW}}\to0
\qquad\text{as \ntoo}
\end{align}
for every $m\in\bbZ\setminus\set0$.
We have
\begin{align}
\E e^{2\pi\ii m\frax{nW}}
=
\E e^{2\pi\ii m{nW}}
=\gf(2\pi mn),
\end{align}
where $\gf$ is the characteristic function of $W$, thus \eqref{FT} follows
by the Riemann-Lebesgue lemma.
\end{proof}

\begin{proof}[Proof of lower bounds in \refT{T1} and \refC{C1}]
(This proof holds for any $q\ge2$.)
Note
that $W_1$ is a continuous random variable while
$\yn$ is integer-valued.
Hence, in any coupling of $\yn$ and $W_1$, 
\begin{equation}\label{pyx}
\P\bigpar{\abs{\yn-nW_1}>\tfrac{1}4}
\ge \P\bigpar{\tfrac14<\frax{nW_1}<\tfrac34}
\ge c,   
\end{equation}
for some $c>0$ and all $n\ge1$,
since the last probability in \eqref{pyx} 
is strictly positive for every $n$ and
converges to $\frac12$ as \ntoo{} by \refL{LLB}.
 It follows that 
$\ellp(\yn,nW_1)\ge\ellx1(\yn,nW_1)\ge c/4$.
The lower bound in \eqref{ayb} follows. The proof of the lower bound in
\eqref{axb} is similar, and the lower bounds in \eqref{ax} and \eqref{ay}
follow trivially from these. 
\end{proof}

\section{Proof of \refT{T1}: $q>2$}\label{Spfq}

We show here how the general case of \refT{T1} follows from the case $q=2$
shown in \refS{Spf2}.
Recall that the lower bounds already are proved for general $q$ in \eqref{pyx}.
Hence, it suffices to show the upper bounds. 

\begin{proof}[Proof of upper bounds in \refT{T1} for $q>2$]
We may still assume $a=1$ for simplicity.
We focus on the version \eqref{a2y}; this implies
\eqref{a2x} and \eqref{ax}--\eqref{ay} as said in the introduction.

We use induction on $q$.
Thus, let $q>2$, and assume that \eqref{a2y} holds when the number of
colours is $q-1$, and when it is 2.

We call the first colour \emph{white}, and all other colours \emph{dark}.
It is then obvious from the definition of the urn, 
in words or by \eqref{polya},
that if we are colour-blind
and only notice whether balls are white or dark, we obtain  another \Polya{}
urn, with two colours. Formally, we define
$\vX_n':=(X_{n,2},\dots,X_{n,q})$ and $X_n':=|\vX_n'|$, and then 
$(X_{n,1},X_n')$ is a \Polya{} urn with initial state $\vx_0'=(x_1,x')$
where $x':=\sum_{j\neq1} x_j$, and the same $a=1$.

Let also $\vY'_n:=(Y_{n,2},\dots,Y_{n,q})=\vX_n'-\vx_0'$ and
$Y_n':=|\vY_n'|$, the number of times in the $n$ first draws that a dark
ball is drawn.
Then, it follows from the case $q=2$ of \eqref{a2y} %\refT{T1} 
 applied to the urn with
white and dark balls that 
if $Z\sim\Beta(x_1,x')$, then, for each $n\ge1$,
\begin{equation}\label{ox1}
  \elloo\bigpar{Y_{n,1},nZ}\le C.
\end{equation}
In other words, there exists
a  random variable $\Zx\sim\Beta(x_1,x')$ such that \as{}
$\bigabs{Y_{n,1}-n\Zx}\le C$, \ie,
\begin{equation}\label{ox}
%  \bigabs{Y_{n,1}-n\Zx}\le C.
Y_{n,1}=n\Zx+\Ooo(1),
\end{equation}
where as in the rest of this proof $\Ooo(1)$ denotes a random variable that is
\as{} bounded by some constant $C$.
(These random variables, and other variables in the proof such as  $\Zx$,
 may depend on $n$, but the bound $C$ does not.)
Note also that $Y_{n,1}+Y_n'=|\vY_n|=n$, and thus
\begin{equation}\label{ox2}
%  \bigabs{Y_{n,1}-n\Zx}\le C.
Y_n'=n-Y_{n,1}=n(1-\Zx)+\Ooo(1).
\end{equation}

Moreover, it also follows from the definition that, conditioned on $Y_n'=m$,
the vector $\vX_n'$ is distributed as 
the $m$:th stage of
a \Polya{} urn with $q-1$ colours and
initial state $\vX'_0$.
Consequently, using the induction hypothesis on this urn,
if $\vW'\sim\Dir(x_2,\dots,x_q)$, then
$\elloo\bigpar{(\vY'_n\mid Y_n'=m),m\vW'}\le C$.
This means that for every $m\in\set{0,\dots,n}$ 
there exists a random vector $\vW^*_m$
defined on the event $\set{Y_n'=m}$ such that
\begin{equation}
  \label{cam1}
\bigpar{\vW_m^*\mid Y_n'=m}\eqd \vW'
\end{equation}
and, if $Y_n'=m$, then \as{}
\begin{equation}\label{cam2}
  |\vY_n'-m\vWx_m|\le C.
\end{equation}
 Define the random vector $\vWx$ on our probability space $\gO$ by
$\vWx:=\vWx_{Y_n'}$, \ie,
$\vWx=\vWx_m$ when $Y_n'=m$. Then \eqref{cam1} means that 
$\vWx\eqd \vW'\sim\Dir(x_2,\dots,x_q)$,
and that $\vWx$ is independent of $Y_n'$, while \eqref{cam2} means that
$|\vY_n'-Y_n'\vWx|\le C$, or equivalently
\begin{equation}\label{isaac}
\vY_n'=Y_n'\vWx+\Ooo(1).
\end{equation}
Note also that since $\vWx$ is independent of $Y_n'$, and thus of $Y_{n,1}$,
we may assume that $\Zx$ in \eqref{ox}--\eqref{ox2} is chosen to be
independent of $\vWx$.

Combining \eqref{isaac} and \eqref{ox}--\eqref{ox2}, we obtain,
recalling that the Dirichlet distribution is supported on a bounded set,
\begin{equation}\label{puh}
\begin{split}
  \vY_n&=\bigpar{Y_{1,n},\vY'_n}
=\bigpar{Y_{1,n},Y_n'\vWx}+\Ooo(1)
\\&
=\bigpar{n\Zx,n(1-\Zx)\vWx}+\Ooo(1)
=n\bigpar{\Zx,(1-\Zx)\vWx}+\Ooo(1).
\end{split}
\end{equation}
Since $\Zx\sim\Beta(x_1,x')$ and $\vWx\sim\Dir(x_2,\dots,x_q)$, with $\Zx$
and $\vWx$ independent,
$\vW:=\bigpar{\Zx,(1-\Zx)\vWx}\sim\Dir(x_1,\dots,x_q)$ by \eqref{dirr}.
Consequently, \eqref{puh} shows that
$\elloo\bigpar{\vY_n,n\vW}\le\normoo{\vY_n-n\vW}=O(1)$
with $\vW\sim\Dir(x_1,\dots,x_q)$,
which completes the induction.

This completes the proof of \refT{T1}.
\end{proof}

\begin{proof}
  [Proof of \refC{C1}]
The upper bounds follow by \refT{T1}, and the lower bounds are already
proved above
by \eqref{pyx}
and the corresponding argument for $X_{n,i}$.
\end{proof}

\section{Proofs of \refTs{TKS} and \ref{TL}}\label{SKS}

\begin{proof}[Proof of \refT{TKS}]
We may assume $a=1$, and 
by symmetry, it suffices to consider $i=1$.
Furthermore, by regarding the colours as
white or dark as in \refS{Spfq}, it suffices to consider $q=2$.
Let, as in \refS{Spf2},
$\ga:=x_1$ and $\gb:=|\vx_0|-x_1$, so \eqref{rho} becomes
\begin{align}
  \label{rhoab}
\rho=\min\set{\ga,\gb,1}.
\end{align}

For simplicity, we consider only \eqref{tksy}. 
Similar arguments yield \eqref{tksx}; note that by \eqref{xy}, 
%assuming as we may $a=1$,
\begin{align}
  \dks\biggpar{\frac{X_{n,1}}{|\vX_n|},W_1}
&=
  \dks\biggpar{\frac{Y_{n,1}+x_1}{n+|\vx_0|},W_1}
=
  \dks\biggpar{\frac{Y_{n,1}+x_1}n,\frac{n+|\vx_0|}n W_1}
\notag\\&
=
  \dks\biggpar{\frac{Y_{n,1}}n,\Bigpar{1+\frac{|\vx_0|}n} W_1-\frac{x_1}{n}}
.\end{align}

Since $Y_{n,1}$ is integer-valued,
we have, for $x\in[k/n,(k+1)/n)$ with integer $k$,
recalling the definitions \eqref{Pnk}, \eqref{Qnk}, \eqref{Rnk},
\begin{align}\label{kula}
  \P(Y_{n,1}/n \le x)-\P(W_1\le x)
&=\P(Y_{n,1}\le k)-\P(W_1\le x)  
\notag\\&
=R_{n,k} -\int_{k/n}^x f_{\ga,\gb}(x)\dd x . 
\end{align}
The integral in \eqref{kula} is less than $\gD Q_{n,k+1}$, and thus,
recalling also \eqref{dks},
\begin{align}\label{kulb}
  \dks\bigpar{\xfrac{Y_{n,1}}n,W_1}
\le\sup_k |R_{n,k}| + \sup_k|\gD Q_{n,k}|.
\end{align}
Furthermore, by taking $x=k/n$ in \eqref{kula},
\begin{align}\label{kulc}
\sup_k |R_{n,k}| \le  \dks\bigpar{\xfrac{Y_{n,1}}n,W_1}
\end{align}
and by instead letting $x\upto (k+1)/n$, 
\begin{align}\label{kuld}
\sup_k |R_{n,k}-\gD Q_{n,k+1}| \le  \dks\bigpar{\xfrac{Y_{n,1}}n,W_1}.
\end{align}
Hence, combining \eqref{kulc} and \eqref{kuld},
\begin{align}\label{kule}
\sup_k |\gD Q_{n,k+1}| \le  2\dks\bigpar{\xfrac{Y_{n,1}}n,W_1}.
\end{align}

It follows from \eqref{Qnk} and \eqref{fab} that
\begin{align}
  \gD Q_{n,1}& \ge c n^{-\ga},
\\
  \gD Q_{n,n}& \ge c n^{-\gb},
\\
  \gD Q_{n,\floor{(n+1)/2}}& \ge c/n.
\end{align}
Consequently, \eqref{kule} yields the lower bound, using \eqref{rhoab},
\begin{align}\label{kulex}
\dks\bigpar{\xfrac{Y_{n,1}}n,W_1} \ge 
c \max\bigpar{ n^{-\ga},n^{-\gb},n\qw}
=c n^{- \rho}.
\end{align}

For an upper bound, we similarly have, see \eqref{gDQ}--\eqref{gDQ1},
\begin{align}
  \max_k \gD Q_{n,k}
\le C \max_{1\le k\le n-1} n\qw f_{\ga,\gb}\Bigparfrac{k}{n}
\le C n^{- \rho}.
\end{align}
Furthermore, \refL{L2} yields
\begin{align}
  \max_k| R_{n,k}|
\le C n^{- \min\set{\ga,\gb,1}}
= C n^{- \rho}.
\end{align}
Consequently, the upper bound in \eqref{tksy} follows from \eqref{kulb}.
\end{proof}

\begin{proof}[Proof of \refT{TL}]
We use the same simplifying assumptions and notations as in the proof of
\refT{TKS}. 
Again, we consider only $Y_{n,1}/n$; the proof for $X_{n,1}/|\vX_n|$ is similar.
\refL{L3} implies that
\begin{align}\label{bja}
\dl\biggpar{\frac{Y_{n,1}}{n},W_1} \le \frac{K}{n}.  
\end{align}

Conversely, suppose that 
$\dl\bigpar{\xfrac{Y_{n,1}}{n},W_1} < \frac{1}{4}n$.
Take any $\gd$ such that
$\dl\bigpar{\xfrac{Y_{n,1}}{n},W_1} < \gd<\frac{1}{4}n$, and 
let $w_0:=\floor{n/2}/n$.
Then, by \eqref{dl} and the fact that $Y_{n,1}$ takes only integer values,
\begin{align}\label{kruka}
  F_{W_1}\Bigpar{w_0+\frac{3}{4n}}&
\le F_{Y_{n,1}/n}\Bigpar{w_0+\frac{3}{4n}+\gd}+\gd
= F_{Y_{n,1}/n}\Bigpar{w_0+\frac{1}{4n}-\gd}+\gd
\notag\\&
\le   F_{W_1}\Bigpar{w_0+\frac{1}{4n}} +2\gd.
\end{align}
We have $W_1\sim \Beta\bigpar{\ga,\gb}$ for some $\ga,\gb>0$,
and thus
\eqref{fab} yields
\begin{align}\label{krubb}
  F_{W_1}\Bigpar{w_0+\frac{3}{4n}} -F_{W_1}\Bigpar{w_0+\frac{1}{4n}} 
=\int_{w_0+\frac{1}{4n}} ^{w_0+\frac{3}{4n}} f_{\ga,\gb}(x)\dd x
\ge \frac{c}n.
\end{align}
Together, \eqref{kruka} and \eqref{krubb} yield $2\gd\ge c/n$.
Hence, if $\dl\bigpar{\xfrac{Y_{n,1}}{n},W_1} < \frac{1}{4}n$, then
$\dl\bigpar{\xfrac{Y_{n,1}}{n},W_1} \ge \frac{c}{2}n$, and thus,
in any case,
\begin{align}\label{bu}
  \dl\biggpar{\frac{Y_{n,1}}{n},W_1} \ge \frac{\min\set{1/4,c/2}}{n}.  
\end{align}
This together with \eqref{bja} completes the proof.
\end{proof}

\section*{Acknowledgement}
This work was partly carried out during a visit to the 
Isaac Newton Institute for Mathematical Sciences
during the programme 
Theoretical Foundations for Statistical Network Analysis in 2016
(EPSCR Grant Number EP/K032208/1)
and was partially supported by a grant from the Simons foundation, and
a grant from
the Knut and Alice Wallenberg Foundation.
I thank Ralph Neininger for asking me the problem, and Gesine Reinert for
helpful comments.

\newcommand\AAP{\emph{Adv. Appl. Probab.} }
\newcommand\JAP{\emph{J. Appl. Probab.} }
\newcommand\JAMS{\emph{J. \AMS} }
\newcommand\MAMS{\emph{Memoirs \AMS} }
\newcommand\PAMS{\emph{Proc. \AMS} }
\newcommand\TAMS{\emph{Trans. \AMS} }
\newcommand\AnnMS{\emph{Ann. Math. Statist.} }
\newcommand\AnnPr{\emph{Ann. Probab.} }
\newcommand\CPC{\emph{Combin. Probab. Comput.} }
\newcommand\JMAA{\emph{J. Math. Anal. Appl.} }
\newcommand\RSA{\emph{Random Struct. Alg.} }
\newcommand\ZW{\emph{Z. Wahrsch. Verw. Gebiete} }
\newcommand\DMTCS{\jour{Discr. Math. Theor. Comput. Sci.} }

\newcommand\AMS{Amer. Math. Soc.}
\newcommand\Springer{Springer-Verlag}
\newcommand\Wiley{Wiley}

\newcommand\vol{\textbf}
\newcommand\jour{\emph}
\newcommand\book{\emph}
\newcommand\inbook{\emph}
\def\no#1#2,{\unskip#2, no. #1,} %(typeset after year) 
\newcommand\toappear{\unskip, to appear}

\newcommand\arxiv[1]{\texttt{arXiv:#1}}
\newcommand\arXiv{\arxiv}

\def\nobibitem#1\par{}

\end{document}